\documentclass[a4paper]{amsart}
\usepackage{latexsym}\usepackage{ifthen}\usepackage[leqno]{amsmath}
\usepackage{enumerate}\usepackage{calc}\usepackage{hyphenat}
\usepackage{amstext,amsbsy,amsopn,amsthm,amsgen}
\usepackage{amsfonts,amscd,amsxtra,upref}
\usepackage{graphicx}
\usepackage{caption}
\usepackage{float}
\usepackage{hyperref}
\restylefloat{table}
\usepackage{array}
\usepackage{mathrsfs}\usepackage{euscript}\usepackage{amssymb}

\newtheorem{thm}{Theorem}[section]
\newtheorem{cor}[thm]{Corollary}
\newtheorem{lem}[thm]{Lemma}

\newtheorem{rem}[thm]{Remark}

\def\bbb{\mathbb}

\renewcommand{\phi}{\varphi}

\newcommand{\N}{\bbb{N}}

\newcommand{\ITE}[3]{\ifthenelse{#1}{#2}{#3}}\newcommand{\ITEE}[4][]{\ITE{\equal{#2}{#3}}{#4}{#1}}
\ITE{\isundefined{\texorpdfstring}}{\newcommand{\texorpdfstring}[2]{#1}}{}

\newcommand{\myData}[1][]{
	\author[D.\ Burek]{Dominik Burek}
	\author[B.\ \.{Z}mija]{B{\l}a\.{z}ej \.{Z}mija}
	\address{\ITEE{#1}{*}{B.\ \.{Z}mija{}\\{}}
			Faculty of Mathematics and Computer Science\\{} Jagiellonian University\\{}
			\L{}ojasiewicza 6\\{}30-348 Krak\'{o}w\\{}Poland}
	\email{dominik.burek@doctoral.uj.edu.pl}
	\email{blazejz@poczta.onet.pl}
}

\begin{document}

\title[A new upper bound for numbers with the Lehmer property\ldots]{A new upper bound for numbers with the Lehmer property and its application to repunit numbers}
\myData

\begin{abstract}
A composite positive integer $n$ has the \textsl{Lehmer property} if $\phi(n)$ divides $n-1,$ where $\phi$ is an Euler totient function. In this note we shall prove that if $n$ has the Lehmer property, then $n\leq 2^{2^{K}}-2^{2^{K-1}}$, where $K$ is the number of prime divisors of $n$. We apply this bound to repunit numbers and prove that there are at most finitely many numbers with the Lehmer property in the set 
$$
\left\{\frac{g^{n}-1}{g-1}\ \bigg|\ n,g\in\N,\ \nu_{2}(g)+\nu_{2}(g+1)\leq L\ \right\},
$$
where $\nu_{2}(g)$ denotes the highest power of $2$ that divides $g$, and $L\geq 1$ is a fixed real number.
\end{abstract}

\subjclass[2010]{Primary 11N25; Secondary 11B39}
\keywords{Lehmer numbers, repunit, upper bound}
\maketitle

\section{Introduction}

In 1932 Lehmer conjectured that if $\phi(n)\mid n-1$, then $n$ has to be a prime number. A composite positive integer satisfying that divisibility is called \textsl{Lehmer number} or number with \textsl{the Lehmer property}.

No Lehmer number is known, although there are some partial results. Pomerance showed in \cite{Pom} that if $n$ satisfies the Lehmer property, then $n$ is squarefree and $n<K^{2^{K}}$, where $K$ denotes the number of prime divisors of $n$. Moreover Renze in \cite{Re} gave a bound $K\geq 15$. If additionally $3\mid n,$ then $K \geq 40\cdot 10^6$ and $n> 10^{36\cdot 10^7}.$

The upper bound given by Pomerance was a crucial step in proofs of many results concerning the existence of Lehmer numbers in certain sequences, such as the Fibonacci sequence \cite{Luca}, Pell numbers \cite{FaLu} or Cullen numbers \cite{GrauLu,KimOh}. 

In this note we prove that $n\leq 2^{2^{K}}-2^{2^{K-1}}$. This new bound allows us to get new results about repunit numbers with Lehmer property, the topic that was studied earlier by Cilleruelo and Luca in \cite{CiLuca}.

\subsection*{Acknowledgments} We are grateful to Prof. S{\l}awomir Cynk and Prof. Maciej Ulas for their helpful remarks and suggestions.

\section{A new upper Bound}

The key ingredient in the proof of the upper bound for numbers with the Lehmer property is the following lemma given by Nielsen in his proof of the bound for odd perfect numbers.

\begin{lem}[Nielsen, \cite{Niels}]
Let $r,a,b\in \N$ and $x_{1},\ldots ,x_{r}$ be integers such that $1<x_{1}<x_{2}<\ldots <x_{r}$ and
\begin{equation}\label{equ1}
\prod_{j=1}^{r}\left(1-\frac{1}{x_{j}}\right)\leq \frac{a}{b}<\prod_{j=1}^{r-1}\left(1-\frac{1}{x_{j}}\right).
\end{equation}
Then
\begin{equation}\label{equ2}
a\prod_{j=1}^{r}x_{j}\leq (a+1)^{2^{r}}-(a+1)^{2^{r-1}}.
\end{equation}
\end{lem}

The following theorem is a simple consequence of the above lemma.

\begin{thm}\label{ThmUp}
If $n$ has the Lehmer property, then
$$
n\leq 2^{2^{K}}-2^{2^{K-1}},
$$ 
where $K$ denotes the number of prime divisors of $n$.
\end{thm}
\begin{proof}
It is known (see \cite{CoHa}), that if $n$ satisfies the Lehmer property, then $n$ is odd and square-free. Let us write $n=p_{1}\cdot\ldots\cdot p_{K}$, where $p_{1}<p_{2}<\ldots <p_{K}$. Then
\begin{align*}
\prod_{j=1}^{K}\left(1-\frac{1}{p_{j}}\right)=\frac{\phi (n)}{n}<\frac{\phi (n)}{n-1}.
\end{align*}
Moreover,
\begin{align*}
\frac{\frac{\phi (n)}{n-1}}{\displaystyle \prod_{j=1}^{K-1}\left(1-\frac{1}{p_{j}}\right)}=\frac{n \displaystyle \prod_{j=1}^{K}\left(1-\frac{1}{p_{j}}\right)}{\left(n-1\right) \displaystyle \prod_{j=1}^{K-1}\left(1-\frac{1}{p_{j}}\right)}=\frac{1-\frac{1}{p_{K}}}{1-\frac{1}{n}}<1.
\end{align*}
Thus
\begin{align*}
	\prod_{j=1}^{K}\left(1-\frac{1}{p_{j}}\right)<\frac{\phi (n)}{n-1}<\prod_{j=1}^{K-1}\left(1-\frac{1}{p_{j}}\right).
\end{align*}
Hence, the inequality (\ref{equ1}) is satisfied for $x_{j}=p_{j}$, $r=K$, $a=1$ and $\displaystyle b=\frac{n-1}{\phi (n)}.$ From (\ref{equ2}) we get
\begin{align*}
n=p_{1}\cdot\ldots\cdot p_{K}\leq 2^{2^{K}}-2^{2^{K-1}},
\end{align*}
and theorem follows.
\end{proof}

\section{Repunit Numbers}

The upper bound obtained in the previous section allow us to improve a result concerning repunit numbers, i.e., numbers of the form $\frac{g^{n}-1}{g-1}$ for integers $g\geq 2$ and $n\geq 1.$  The following is true:

\begin{thm}[Cilleruelo, Luca \cite{CiLuca}]\label{CiLu}
	For each fixed $g>1$ there are only finitely many effectively computable positive integers $n$ such that $\displaystyle \frac{g^{n}-1}{g-1}$ is a Lehmer number.
\end{thm}

For now suppose that $g$ is even integer. Let $L\geq 1$ be a fixed real number and define
$$
A(L)=\left\{a_{g,n}=\frac{g^{n}-1}{g-1}\ \bigg|\ g,n\in\N,\ 1\leq\nu_{2}(g)\leq L \right\}, 
$$ where $\nu_{2}(g)$ denotes the $2$-adic valuation of $g$, that is, the exponent of $2$ in the prime factorization of $g$.
 
We offer the following generalization of Theorem \ref{CiLu}:

\begin{thm}\label{RepEven}
	If $a_{g,n}\in A(L)$ satisfies the Lehmer property, then $g<2^{2^{L}}$ and $L\geq 15$. In particular there are only finitely many Lehmer numbers in $A(L)$.
\end{thm}
\begin{proof}
Consider $g$ and $n$ such that $\displaystyle a_{g,n}\in A(L)$ satisfies the Lehmer property. Then obviously $n\geq 2$. Let $K$ denotes the number of prime factors of $a_{g,n}$. From \cite{Re} we know that $K\geq 15.$

Now
\begin{align*}
2^{K} \mid \phi(a_{g,n})\mid a_{g,n}-1=g\cdot \frac{g^{n-1}-1}{g-1}=g\cdot \left(g\cdot\frac{g^{n-2}-1}{g-1}+1\right).
\end{align*}
Parity of $g$ leads to $2^{K}\mid g$ and thus 
$$
15\leq K\leq \nu_{2}(g)\leq L.
$$ 
By Theorem \ref{ThmUp} we get
\begin{align*}
a_{g,n}<2^{2^{K}}\leq 2^{2^{L}}.
\end{align*}
On the other hand, $a_{g,n}>g^{n-1}$, so
\begin{align*}
n< \frac{2^{L}\log 2}{\log g}+1.
\end{align*}
The right hand side goes to $1$ as $g\to \infty$ and for $g\geq 2^{2^{L}}$ we have $n<2$, so $n=1$ --- a contradiction. Thus $g< 2^{2^{L}}$ and we get the claim.
\end{proof}

In the case of odd numbers $g$ it is more convenient to bound the $2$-adic valuation of $g+1$ instead of $g$. Let $L\geq 1$ and consider the set
$$
B(L):=\left\{b_{g,n}=\frac{g^{n}-1}{g-1}\ \bigg|\ g,n\in\N,\ 1\leq\nu_{2}(g+1)\leq L \right\}.
$$
The following holds:

\begin{thm}\label{RepOdd}
	If $b_{g,n}\in B(L)$ satisfies the Lehmer property, then $g<2^{2^{L-1}}$. In particular there are only finitely many Lehmer numbers in $B(L)$.
\end{thm}

\begin{proof}
Consider $g$ and $n$ such that $\displaystyle b_{g,n}\in B(L)$ satisfies the Lehmer property and let $K$ denote the number of its prime divisors. Let us write $n-1=2^{s}(2m+1)$, where $s=\nu_{2}(n-1)\geq 1$ and $m$ is a non-negative integer. Then
\begin{align*}
b_{g,n}= & g^{n-1}+\ldots +g+1=g\cdot(g^{n-2}+\ldots +1)+1=g\cdot\frac{g^{2m+1}-1}{g-1}\cdot\frac{g^{n-1}-1}{g^{2m+1}-1}+1 \\ 
 = & g\cdot(g^{2m}+\ldots +g+1)\cdot\frac{g^{n-1}-1}{g^{2m+1}-1}+1.
\end{align*}
Moreover,
\begin{align*}
2^{K}\mid \phi (b_{g,n})\mid b_{g,n}-1=g\cdot (g^{2m}+\ldots +g+1)\cdot \frac{g^{n-1}-1}{g^{2m+1}-1}.
\end{align*}
Since $g(g^{2m}+\ldots +g+1)$ is odd, we have 
\begin{align}\label{dupa}
K\leq \nu_{2}\left(\frac{g^{n-1}-1}{g^{2m+1}-1}\right).
\end{align}
Let $h:=g^{2m+1}$. Then
\begin{align*}
\frac{g^{n-1}-1}{g^{2m+1}-1}&=1+h+h^{2}+\ldots +h^{2^{s}-1}=\\&=(1+h)(1+h^{2})(1+h^{4})\ldots (1+h^{2^{s-2}})(1+h^{2^{s-1}}).
\end{align*}
Observe, that for any odd number $t$ we have $\nu_{2}(t^{2}+1)=1$. Thus
\begin{align*}
\nu_{2}\left(\frac{g^{n-1}-1}{g^{2m+1}-1}\right)&=\nu_{2}(g^{2m+1}+1)+s-1=\\&=
\nu_{2}\left((g+1)(g^{2m}-g^{2m-1}+\ldots -g+1)\right)+s-1=\nu_{2}(g+1)+s-1.
\end{align*}
Thus from (\ref{dupa}) we get
$
K\leq \nu_{2}(g+1)+\nu_{2}(n-1)-1.
$
Also, from Theorem \ref{ThmUp} we infer that
\begin{align*}
g^{n-1}<b_{g,n}<2^{2^{K}}\leq 2^{2^{\nu_{2}(g+1)+\nu_{2}(n-1)-1}},
\end{align*}
so
\begin{align*}
n-1<\frac{\log 2}{\log g}\cdot 2^{\nu_{2}(g+1)-1+\nu_{2}(n-1)}\leq \frac{\log 2}{\log g}\cdot 2^{\nu_{2}(g+1)-1}(n-1)\leq \frac{\log 2}{\log g}\cdot 2^{L-1}(n-1).
\end{align*}
Equivalently, $g<2^{2^{L-1}}$. 

Finally, the finiteness statement follows from Theorem \ref{CiLu}.

\end{proof}

\begin{rem}
From the proof of previous results we see that if $b_{g,n}$ has the Lehmer property and $n-1=2^{s}(2m+1)$, then 
\begin{align*}
2m+1<\frac{\log 2}{\log g}\cdot 2^{L-1}.
\end{align*}
\end{rem}

From Theorems \ref{RepEven} and \ref{RepOdd} we get effective bounds for $n$ and $g$ in terms of $L$. In the case of even $g$ they are given explicitly in the proof. In the case of odd $g$ we showed that $g<2^{2^{L-1}}$ and used the result of Luca and Cilleruello, where the explicit bound for $n$ is proved. Thus, gathering these results we have the following obvious:

\begin{cor}\label{pimpek}
For a fixed $L\geq 1$, there are only finitely many numbers with the Lehmer property in the set
$$
\left\{\frac{g^{n}-1}{g-1}\ \bigg|\ g,n\in\N,\ \nu_{2}(g)+\nu_{2}(g+1)\leq L \right\}.
$$
All of them are effectively computable.
\end{cor}

Corollary \ref{pimpek} improves Theorem \ref{CiLu} because the parameter $g$ can take values from an infinite set. Our idea was to show, that in fact Lehmer numbers can appear only for finitely many values of $g$. This generalization was impossible to get using the bound $n<K^{2^{K}}$ since it implied only the inequality of the form $n<1+\left\lfloor \frac{g\log (\nu_{2}(g))}{\log g}\right\rfloor$ (see the proof of Theorem \ref{CiLu} in the even case) and the expression on the right is unbounded when $g$ goes to infinity.

\end{document}